\documentclass{amsart}
\usepackage{amsmath, amsfonts, amssymb,amsthm}
\usepackage{amstext}
\usepackage{mathrsfs}
\usepackage[dvipsnames]{xcolor}
\usepackage{tikz}
\usepackage[all]{xy}
\usepackage{enumerate}
\usepackage{mathrsfs}
\usetikzlibrary{er,positioning}
\allowdisplaybreaks
\usepackage{lineno}

\usepackage{comment}

\theoremstyle{definition}
\newtheorem{definition}{Definition}[section]

\newtheorem{theorem}[definition]{Theorem}
\newtheorem{lemma}[definition]{Lemma}
\newtheorem{proposition}[definition]{Proposition}

\newcommand{\sq}[1]{\ifx#1([\else\ifx#1)]%
  \else\message{invalid use of "sq"}\fi\fi}

\newcommand{\C}{\mathbb{C}}
\newcommand{\PP}{\mathbb{P}}

\DeclareMathOperator{\ord}{ord}

\DeclareMathSymbol{\idot}{\mathbin}{operators}{`\.}
\allowdisplaybreaks
\begin{document}
\title[Meromorphic maps from ${\Bbb C}^{\ell}$ into semi-abelian varieties and general projective varieties]{Meromorphic maps from ${\Bbb C}^{p}$ into   semi-abelian varieties and general projective varieties}
\author{Zhe Wang}
\address{Department of Mathematics\newline
	\indent University of Houston\newline
	\indent Houston,  TX 77204, U.S.A.} 
\email{zwang224@cougarnet.uh.edu}

\begin{abstract}
In \cite{W. Stoll}, W. Stoll proposed a  method of studying holomorphic functions of several complex variables by reducing them to  one variable through fiber integration.
In this paper, we use this method to extend some important  Nevanlinna-type results  for holomorphic curves into projective varieties to meromorphic maps from ${\Bbb C}^{p}$ to projective varieties.
This includes Bloch's theorem and Noguchi-Winkelmann-Yamanoi's Second Main Theorem for holomorphic maps into semi-abelian varieties intersecting an effective divisor, as well as  
 Huynh-Vu-Xie's Second Main Theorem for meromorphic maps into projective space intersecting with a generic hypersurface with sufficiently high degree. \end{abstract}
 \thanks{2010\ {\it Mathematics Subject Classification.} 32H30, 32A22, 32Q45.}  
\keywords{Meromorphic maps of several complex variables, Bloch's theorem, Semi-abelian varieties, Second Main Theorem, Nevanlinna Theory}

\baselineskip=16truept \maketitle \pagestyle{myheadings}
\markboth{}{The Second Main Theorem}
\section{Introduction}
Recently, there have been several works in the study of meromorphic maps from ${\Bbb C}^{p}$ to  projective varieties, where $p\ge 1$ is an integer.
Notably,  A. Etesse \cite{Ete} developed the theory of  the Green-Griffiths jet differentials of $p$-germs, Q. Cai, M. Ru, and C. J. Yang  \cite{CaiRuYang2} used the result of Etesse to establish several defect relations for meromorphic maps from ${\Bbb C}^{p}$ to  projective varieties. See also \cite{PR}, and for earlier results, \cite{HuYang}.
In this paper, we study meromorphic maps from ${\Bbb C}^{p}$ into projective varieties by reducing the $p$-variables  to  one variable through fiber integration, following the method of  W. Stoll \cite{W. Stoll}.  

We first focus on the meromorphic maps  from ${\Bbb C}^{p}$ into semi-abelian varieties. The key result we used is the 
following (see Corollary 5.1.9 in \cite{J. Noguchi and J. Winkelmann}): {\it A semi-torus $A$ contains only  countably many sub-semi-tori}.
We also note that every meromorphic map from ${\Bbb C}^{p}$ into a semi-abelian variety is indeed holomorphic. Therefore, we only state the results for holomorphic  maps from ${\Bbb C}^{p}$ into  semi-abelian varieties.

Our first result is to establish Bloch's theorem for several complex variables.
\begin{theorem}[Bloch's theorem for several variables]\label{Bloch}
Let $A$ be a semi-abelian variety and let 
$f : \mathbb{C}^p \to A$ be a holomorphic map. 
Then the Zariski closure of $f(\C^p)$ is a translate of a semi-abelian subvariety of $A$.

\end{theorem}
Next we extend the result of  Noguchi-Winkelmann-Yamanoi's Second Main Theorem (see Theorem~6.4.1 in \cite{J. Noguchi and J. Winkelmann} or \cite{NWY}; see also \cite{NWY2})
 to the case of several complex variables.
\begin{theorem}\label{SEMIABELIAN}
Let $A$ be a semi-torus with equivariant compactification $\overline{A}$.  
Let $f:\,\C^p \to A$ be a holomorphic map with Zariski dense image in $\overline{A}$.  
Let $D$ be an effective divisor on $A$ which extends to a divisor $\overline{D}$ on $\overline{A}$. Then, possibly after changing the compactification $\overline{A}$ (depending only on $D$ and independent of $f$), there exists a positive integer $k_0$, depending on $f$ and $D$, such that
\[
T_f(r, \overline{D})\leq N^{(k_0)}_{f}(r, D)+S_f(r, \overline{D}) .
\]
\end{theorem}

The second part of this paper is to study meromorphic maps from ${\Bbb C}^{p}$ to  projective varieties, where $p\ge 1$ is an integer.  We first establish the following version of   the Second Main Theorem with truncation level one (see Theorem 1.4 in \cite{CaiRuYang2}). For related results, see also \cite{HVX}.
  \begin{theorem}\label{smtXie}
	Let $X$ be a smooth projective variety of dimension $n$.
	Let $D$ be a normal crossing divisor on $X$ and $A$ be an ample line bundle on $X$. Let  
	\begin{equation*}
	\mathcal P \in H^0(X,E^{GG}_{k,m}(T_X^*(\log D)) \otimes  A^{-\tilde m}  ).
	\end{equation*}
	Let $f:  {\Bbb C}^p \rightarrow X$ be a meromorphic map with $f({\Bbb C}^p) \not\subset \text{Supp}(D)$. Assume that  ${\mathcal P}\big(j_k(f_{\vec a})\big) \not \equiv 0$ for some $\vec{a}\in S_p(1)$. Then,
	\begin{equation*}
	\tilde mT_{f}(r,A) \leq mN^{(1)}_f(r, D)+S_{f}(r,A).
	\end{equation*}
	\end{theorem}

To apply Theorem~1.3 in the case $X=\mathbb{P}^{n}(\mathbb{C})$, where 
$D$ is a generic hypersurface of sufficiently high degree $d$ and 
$A=\mathcal{O}(1)$, following Siu’s strategy (see \cite{HCG}, \cite{Siu2015}), we combine the existence 
theorem for logarithmic jet differentials on $\mathbb{P}^{n}(\mathbb{C})$ vanishing on an ample divisor with the technique of slanted vector fields.

More precisely, let $S := \mathbb{P}H^{0}\!\left(\mathbb{P}^{n}(\mathbb{C}),\mathcal{O}(d)\right)$ be the projective parameter space of homogeneous polynomials of degree $d$ in $\mathbb{P}^{n}(\mathbb{C})$ and let $s\in S$ be the point corresponding to the above hypersurface $D$. We view $X=\mathbb{P}^{n}(\mathbb{C})$ as the fiber 
$\mathbb{P}^{n}(\mathbb{C})\times\{s\}\subset \mathbb{P}^{n}(\mathbb{C})\times S$. Let $\mathcal{P}_s$ be a fixed nonzero logarithmic jet differential 
on $\mathbb{P}^{n}(\mathbb{C})\times\{s\}$.  
By the semi-continuity theorem (\cite{AG}, Theorem~12.8), it extends 
to a holomorphic family $\mathcal{P}$ of nonzero logarithmic jet 
differentials parametrized by points of a Zariski open neighborhood 
$U_s$ of $s$ in $S$.
Applying slanted vector fields with low pole order to this holomorphic family then produces new 
jet differentials satisfying the hypotheses of Theorem~1.3. This leads to the following Second Main Theorem for 
meromorphic maps into projective space intersecting a generic hypersurface 
of sufficiently high degree (for details, see \cite{Siu2015}, \cite{HVX}, \cite{DMR}).

\begin{theorem}\label{GENERIC}
Let $D \subset \mathbb{P}^n(\mathbb{C})$ be a generic hypersurface having degree 
\[
d \geq 15 (5n + 1) n^n.
\]
Let $f : \mathbb{C}^p \to \mathbb{P}^n(\mathbb{C})$ be a meromorphic map.  
Assume that $f$ has the Zariski dense image. Then the following estimate holds:
\[
T_f(r) \leq N_f^{(1)}(r, D) + S_f(r).
\]
\end{theorem}

Note that in the above theorem, $f$ is assumed to be algebraically non-degenerate (i.e. $f$ has the Zariski dense image). 
However,  according to  Riedl and Yang \cite{RY},  to derive the hyperbolicity of the complement of  generic hypersurfaces  in ${\Bbb P}^n(\C)$, it can be reduced to studying 
the algebraic degeneracy of the image of  $f$.  For the progress on the hyperbolicity of the complement of  generic hypersurfaces  in ${\Bbb P}^n(\C)$, see \cite{Siu2015}, \cite{Brotbek}, \cite{BD}, and \cite{DMR}.

\section{Notations and preparations}
We recall some notations and definitions. Let $p\ge 1$ be an integer.
For $z = (z_1, \ldots, z_p) \in \mathbb{C}^p$, we define
$
|z| = \bigl(|z_1|^2 + \cdots + |z_p|^2\bigr)^{1/2}.
$
We denote by
$
B_p(r) = \{ z \in \mathbb{C}^p \mid |z| < r \}
$
the open ball in $\mathbb{C}^p$ of radius $r$ centered at the origin, and
$
S_p(r) = \{ z \in \mathbb{C}^p \mid |z| = r \}
$
the boundary of $B_p(r)$.  Let $\iota : S_p(r) \hookrightarrow \mathbb{C}^p$ be the natural inclusion.  
Then the pull-back of the form
\[
\sigma_p = d^c \log |z|^2 \wedge (dd^c \log |z|^2)^{p-1}
\]
on $\mathbb{C}^p \setminus \{0\}$ gives a positive measure on $S_p(r)$ with total measure $1$, i.e.
$
\int_{S_p(r)} \iota^* \sigma_p = 1.
$ We define
$
\nu_p(z) = dd^c |z|^2  \text{  on  } \mathbb{C}^p$.
It  defines a Lebesgue measure on $\mathbb{C}^p$ such that $B_p(r)$ has measure $r^{2p}$.

We now recall the standard notations in Nevanlinna theory (see \cite{book:minru}, \cite{book:minru2}).  
 Let $X$ be a smooth projective variety of dimension $n$ and $D$ be an effective Cartier divisor on $X$.  Denote by $[D]$ the line bundle associated to $D$, and  fix  a Hermitian metric on $[D]$.
 Let $f$ be a meromorphic  map of\ $\mathbb{C}^p$\ into\  $X$, with $f(\mathbb{C}^p)\not\subset \text{Supp}D$.
 
Fix $s>0$. We define the \emph{counting function} of $f$ with respect to $D$  by
\[
N_f(r,s,D) = \int_s^r \frac{dt}{t^{2p-1}} \int_{B_p(t)\cap f^*D} v_p^{p-1}
.\]

Let $f^*D= \sum_{\lambda} \, k_\lambda D_\lambda$ be its irreducible decomposition.
For \(1 \le k \le \infty\), denote $(f^*D)^{(k)}= \sum_{\lambda} \, \min\{k, k_\lambda\} D_\lambda$, the \emph {truncated counting function of $f$ to level \(k\)} with respect to $D$ is defined by
\[
N^{(k)}_f(r,s,D) = \int_s^r \frac{dt}{t^{2p-1}} \int_{B_p(t)\cap (f^*D)^{(k)}} v_p^{p-1}
.\]
Note that $N^{(\infty)}_f(r,s,D)=N_f(r,s,D)$.

The \emph{Weil function} on $X$ with respect to $D$ is defined by
\[
\lambda_D(x) = - \log \| s_D(x) \|,
\]
where $s_D$ is the canonical section of $[D]$. The \emph{proximity function} of $f$ with respect to $D$ is defined by
\[
m_f(r,s,D) = \int_{S_p(r)} \lambda_D(f)~ \sigma_p 
            - \int_{S_p(s)} \lambda_D(f)~\sigma_p. 
\]

The \emph{characteristic function} of $f$ with respect to $D$ is defined by
\[
T_f(r,s,D) = \int_s^r \frac{dt}{t^{2p-1}} \int_{B_p(t)} f^* c_1[D] \wedge v_p^{p-1}.
\]

Note that for \(1 \le k \le \infty\), $T_{f}(r,s,D) $, $m_f(r,s,D)$, $\text{and }N^{(k)}_f(r,s,D)$ only depend on  $s$ up to a bounded  term, so sometimes we denote them as  $T_{f}(r, D) $, $m_f(r, D)$, and $N^{(k)}_f(r,s,D)$ when $s=1$. 

As is customary in Nevanlinna theory, we use the symbol $S_f(r,D)$ for a non-negative small term such that
\[
S_f(r,D)\leq_{exc} O(\log^+T_f(r,D))+O(\log r)+O(1).
\]
where $\leq_{exc}$ stands for the validity of the inequality except for $r$ in exceptional intervals with finite measure.

If $X=\mathbb{P}^n(\mathbb{C})$ and $H\subset \mathbb{P}^n(\mathbb{C})$ is a hyperplane, we write
$
T_f(r,H)=T_f(r),  ~~S_f(r,H)=S_f(r)
$
for simplicity. When $n=1$, we also denote $N_f(r,a)$ by $N(f,r,a)$ for $a \in \mathbb{P}(\mathbb{C})$.

With the above notation, we state the following theorem which is one of the fundamental theorems in Nevanlinna theory.
\begin{theorem}[First Main Theorem]
Let $X$ be a complex projective variety, and let $D$ be an effective Cartier divisor.
Let $f : \mathbb{C}^p \to X$ be meromorphic such that $f(\mathbb{C}^p) \not\subset \operatorname{Supp}(D)$.
Then
\[
T_{f}(r,s,D) = m_f(r,s,D) + N_f(r,s,D).
\]
\end{theorem}

Let $\vec{a}\in \mathbb{C}^p$ with $|\vec{a}| = 1$ and define the map $\iota_{\vec{a}} : \mathbb{C} \to \mathbb{C}^p$ by 
$\iota_{\vec{a}}(z) = z \vec{a}.$
It is easy to see that for almost all $\vec{a}\in S_p(1)$, $f_{\vec{a}} := f \circ \iota_{\vec{a}} : \mathbb{C} \to X$ is holomorphic. 
\begin{lemma}\label{zero measure}
Let $X$ be a smooth projective variety of dimension $n$, and let $D$ be an effective Cartier divisor on $X$. 
Suppose that $f : \mathbb{C}^p \to X$ is a meromorphic map such that 
\[
f(\mathbb{C}^p) \not\subset \operatorname{Supp}(D).
\]
Then, for almost all $\vec{a} \in S_p(1)$, the restriction map $f_{\vec{a}}: \mathbb{C} \to X$
also satisfies 
\[
f_{\vec{a}}(\mathbb{C}) \not\subset \operatorname{Supp}(D).
\]
\end{lemma}

\begin{proof}
  Consider the meromorphic map
\[
\tilde{f} : \mathbb{C}^p \times \mathbb{C} \to X, 
\quad \tilde{f}(v,z) = f(zv).
\]
Then $\tilde{f}(\mathbb{C}^p \times \mathbb{C}) \not\subset \operatorname{Supp}(D)$ and $\tilde{f}^{-1}(\operatorname{Supp}(D))$ is a proper analytic subvariety of a Zariski open subset of $\mathbb{C}^p \times \mathbb{C}$.
Let $ K=\left\{\vec a\in S_p(1)\;\middle|\;
f_{\vec a}(\mathbb C)\subseteq\operatorname{Supp}(D)
\right\}\subseteq S_p(1).
$ We claim that $
m_{S_p(1)}(K) = 0
$, where $m_{S_p(1)}$ is the induced Lebesgue measure from the standard Lebesgue measure $m$ of $\mathbb{C}^p$. 
Indeed, assume $m_{S_p(1)}(K) > 0$ and denote by $\mathbb{C} K$ the set $\left\{a v : a \in \C,\ v \in K \right\}$, then 
$m_{\mathbb{C}^p \times \mathbb{C}}(\mathbb{C} K \times \mathbb{C}) = +\infty$. On the other hand, since
$\mathbb{C} K \times \mathbb{C} \subset \tilde{f}^{-1}(\text{Supp}D)$,
we get 
$
m_{\mathbb{C}^p \times \mathbb{C}} (\mathbb{C} K \times \mathbb{C})
   \le m_{\mathbb{C}^p \times \mathbb{C}} (\tilde{f}^{-1}(\text{Supp}D)) = 0
$.
This is a contradiction.
\end{proof}
From the above Lemma, we know that
$
m_{f_{\vec{a}}}(r,s,D), 
N_{f_{\vec{a}}}(r,s,D), \text{~~and~~} N^{(k)}_{f_{\vec{a}}}(r,s,D)
$
are well defined for $1 \le k \le \infty$, for almost all $\vec{a} \in S_p(1)$.
\begin{lemma}[\cite{W. Stoll}]\label{INTEGRATION}
Let $r>0$ and let $F \in L^1(S_{p}(r))$, then
\[
\int_{S_{p}(r)} F\sigma_p 
= \int_{S_{p}(1)} \left( \int_0^{2\pi} F(re^{i\theta}\vec{a})\,\frac{d\theta}{2\pi} \right) \sigma_p(\vec{a}).
\]  
\end{lemma}

Let $D \subset \C^n$ be an irreducible reduced divisor, defined by a holomorphic function $h$,  
i.e.\ $\{ h = 0 \} = D$.

Then $D = D_{\mathrm{reg}} \cup D_{\mathrm{sing}}$ with $\dim D_{\mathrm{sing}} \leq n-2$.

Consider $\pi : \C^n \setminus \{0\} \to \PP^{n-1}(\C)$ defined by $x \mapsto [x]$.  
Then $\dim \pi(D_{\mathrm{sing}} \setminus \{0\}) \leq n-2$, hence $\pi(D_{\mathrm{sing}} \setminus \{0\})$ has zero measure in $\PP^{n-1}(\C)$.

Now denote ${D_{\mathrm{reg}} \setminus \{0\}}=D_{\mathrm{reg}}^*$ and consider $\pi|_{D_{\mathrm{reg}}^* } : D_{\mathrm{reg}}^*\to \PP^{n-1}(\C)$.  
Note that $D_{\mathrm{reg}}^*$ is a complex manifold in $\C^n \setminus \{0\}$ of dimension $n-1$. 
\begin{lemma}
Let $v \in D_{\mathrm{reg}}^* \subset \C^n \setminus \{0\}$.  
Then $(d\pi|_{D_{\mathrm{reg}}^*})_{v}$ is isomorphic  
if and only if $T_v^{1,0}({D_{\mathrm{reg}}^* }) \oplus \C\{v\} = \C^n$.
\end{lemma}

\begin{proof}
$T_v^{1,0}({D_{\mathrm{reg}}^* })$ and $\C\{v\}$ either satisfy
\[
T_v^{1,0}({D_{\mathrm{reg}}^* }) \oplus \C\{v\} = \C^n,
\quad\text{or}\quad
\C\{v\} \subseteq T_v^{1,0}({D_{\mathrm{reg}}^* }).
\]

Therefore, if we have $\C\{v\} \subseteq T_v^{1,0}({D_{\mathrm{reg}}^* })$ and because $\ker(d\pi)_v = \C\{v\}$, we get $\C\{v\} = \ker(d\pi|_{D_{\mathrm{reg}}^*})_v$.  
Hence $(d\pi)_{D_{\mathrm{reg}}^*}$ is not an isomorphism.

Conversely, if $T_v^{1,0}({D_{\mathrm{reg}}^* }) \oplus \C\{v\} = \C^n$, then again by $\ker(d\pi)_v = \C\{v\}$, we conclude that $(d\pi|_{D_{\mathrm{reg}}^*})_v$ is isomorphic.
\end{proof}

By Sard’s theorem, the set of critical values of $\pi|_{D_{\mathrm{reg}}^*}$ has zero measure in $\PP^{n-1}(\C)$.  
Denote this set by $K$. Then for all $y \in \PP^{n-1}(\C) \setminus (K \cup \pi(D_{\mathrm{sing}} \setminus \{0\}))$, set $(\pi|_{D_{\mathrm{reg}}^*})^{-1}(y) = \{v_1,v_2,\dots\}.$ We have $T_{v_i}^{1,0}({D_{\mathrm{reg}}^* }) \oplus \C\{v\} = \C^n \text{ for } i=1,2,\dots $

Suppose  $T_v^{1,0}({D_{\mathrm{reg}}^* }) \oplus \C\{v\} = \C^n$, then the intersection number of $\C\{v\}$ and $D_{\mathrm{reg}}^*$ at $v$ is $1$, i.e. $i_v(\C\{v\}, D_{\mathrm{reg}}^*) = 1.$
Indeed, let $v \in D_{\mathrm{reg}}^* \subset \C^n \setminus \{0\}$,  
then there exists $h \in \mathcal{O}_{\C^n,v}$ such that $(Z(h))_v = (D_{\mathrm{reg}}^*)_v$ and $\nabla h(v) \neq 0$.
The intersection number is given by $i_v(\C\{v\}, D_{\mathrm{reg}}^*) = \operatorname{ord}_1 h_v(z),$ where $h_v(z) := h(zv)$. We have ${h^{\prime}_v}(1) = v \cdot \nabla h(v)$, and assume ${h^{\prime}_v}(1) = 0$, by
$T_v^{1,0}({D_{\mathrm{reg}}^* }) = \{x \in \C^n \mid x \cdot \nabla h(v) = 0\}.$  We have $v \in T_v^{1,0}({D_{\mathrm{reg}}^* })$, which implies $\C\{v\} \subseteq T_v^{1,0}({D_{\mathrm{reg}}^* })$.  
But this contradicts $T_v^{1,0}({D_{\mathrm{reg}}^* }) \oplus \C\{v\} = \C^n$. Therefore, $h^{\prime}_v(1) \neq 0$, hence
$\operatorname{ord}_1 h_v(z) = 1,$ and consequently $i_v(\C\{v\}, D_{\mathrm{reg}}^*) = 1.$

By the above lemma, we have the following proposition.
\begin{proposition}\label{2.5}
With the above notations, we have
\begin{align*}
\int_{S_p(1)} m_{f_{\vec{a}}}(r,D) \, \sigma_p(\vec{a}) 
&= m_f(r,D), \\
\int_{S_p(1)} T_{f_{\vec{a}}}(r,D) \, \sigma_p(\vec{a}) 
&= T_f(r,D)+ O(\log r), \\
\int_{S_p(1)} N_{f_{\vec{a}}}^{(k)}(r,D) \ \sigma_p(\vec{a})
&= N_f^{(k)}(r,D) + O(\log r), \ \text{for any }1\le k \le\infty.
\end{align*}

\end{proposition}
\begin{proof}
The first equality holds by Lemma \ref{INTEGRATION} and the definition of proximity function. For the second equality, without loss of generality, we assume that $D$ is very ample, since any Cartier divisor can be written as the difference of two ample divisors. Then $D$ induces a canonical embedding $
\iota : X \hookrightarrow \mathbb{P}^N
$
for some $N$, such that
$\iota^*\mathcal{O}_{\mathbb{P}^N}(1) \cong [D].
$ 
Denote by $\omega_{FS}$  the Fubini-Study metric on $\mathbb{P}^N$.  Then:
\begin{align*}
\int_s^r \frac{dt}{t^{2p-1}} \int_{B_p(t)} f^* \iota^* \omega_{FS} \wedge v_p^{p-1} 
&= \int_s^r \frac{dt}{t^{2p-1}} \int_{B_p(t)} f^* c_1[D] \wedge v_p^{p-1} \\[6pt]
&= T_f(r,s,D) + O(1).
\end{align*}
Let a reduced representation of $\iota \circ f: \mathbb{C}^p \rightarrow \mathbb{P}^N$ be $(F_0, \dots, F_N)$, where $F_0,\dots,F_N \in \mathcal{O}(\mathbb{C}^p)$. By the Green-Jensen formula  (see \cite{book:minru}, \cite{book:minru2}), we have
\begin{align*}
&\int_s^r \frac{dt}{t^{2p-1}}
   \int_{B_p(t)} (\iota \circ f)^* \omega_{FS} \wedge \nu_p^{p-1} \\[6pt]
&= \frac{1}{2}\int_{S_p(r)} \log (\sum_{i=0}^N |F_i|^2 ) \,\sigma_p
   \;-\frac{1}{2}\; \int_{S_p(s)} \log (\sum_{i=0}^N |F_i|^2 ) \,\sigma_p\\[6pt]
&= \frac12\int_{S_p(1)}
\left(
\int_{S_1(r)}
\log\Bigl(\sum_{j=0}^N |F_j(\vec{a}z)|^2\Bigr)\sigma_1
-
\int_{S_1(s)}
\log\Bigl(\sum_{j=0}^N |F_j(\vec{a}z)|^2\Bigr)\sigma_1
\right)\sigma_p(\vec{a}).\\
\end{align*}

For almost all \(\vec a\in S_p(1)\), we have $\mathbb{C}^*\vec a \cap I(f)=\emptyset$, where $I(f)=\{F_0= \dots=F_N=0\}$ is the indeterminacy locus of $f$, $F_0(\vec a z),\dots,F_N(\vec a z)$ may have common zeros. Let $m=\min_{0\le j\le N}\ord_0 F_j( z)$. Then $(
\frac{F_0(\vec a z)}{z^m},
\dots,\frac{F_N(\vec a z)}{z^m})$ is a reduced representation of $\iota \circ f_{\vec{a}}$ for almost all  \(\vec a\in S_p(1)\).

We obtain
\begin{align*}
&\int_s^r \frac{dt}{t^{2p-1}}
\int_{B_p(t)}
(\iota \circ f)^* \omega_{FS}\wedge v_p^{p-1} \\
&=
\frac12\int_{S_p(1)}
\Biggl(
\int_{S_1(r)}
\log
\Biggl(
\sum_{j=0}^N
\left|
\frac{F_j(\vec a z)}{z^{m}}
\right|^2
\Biggr)
\sigma_1 
-
\int_{S_1(s)}
\log
\Biggl(
\sum_{j=0}^N
\left|
\frac{F_j(\vec a z)}{z^m}
\right|^2
\Biggr)
\sigma_1
\Biggr)\sigma_p(\vec{a}) \\
&
+
\frac12\int_{S_p(1)}
\left(
\int_{S_1(r)}
\log |z|^{2m}\sigma_1-\int_{S_1(s)}
\log |z|^{2m}\sigma_1\right)
\,\,\sigma_p(\vec{a}) .\\
&=
\int_{S_p(1)}\left(
\int_s^r
\frac{dt}{t}
\int_{B_1(t)}
(\iota\circ f_{\vec a})^\ast\omega_{FS}\right)
\,\sigma_p(\vec a) +O\!\left(\log\frac rs\right)\\
&=
\int_{S_p(1)}
T_{f_{\vec a}}(r,s,D)
\,\sigma_p(\vec a)+O\!\left(\log\frac rs\right).
\end{align*}

Hence
\[
 T_f(r,s,D) = \int_{S_p(1)} T_{f_{\vec{a}}}(r,s,D) \, \sigma_p(\vec{a}) + O(\log\frac rs).
\]

By the First Main Theorem, it follows that
\begin{align}\label{counting}
N_f(r,s,D)
= \int_{S_p(1)} N_{f_{\vec{a}}}(r,s,D)\,\sigma_p(\vec{a})  + O(\log\frac rs).  
\end{align}

We now prove the third identity for $1\le k <\infty$.
Since $f^*D$ is an effective divisor on $\mathbb{C}^p$ with the irreducible decomposition
$
f^*D = \sum_{\lambda \in I} k_\lambda D_\lambda
$ and $\sum_{\lambda \in I} \min\{k, k_\lambda\} D_\lambda$ is a principal divisor on $\mathbb{C}^p$, there is a holomorphic function $g:\mathbb{C}^p \to \mathbb{P}^1(\C)$ such that $g^*(0)=\sum_{\lambda \in I} \min\{k, k_\lambda\} D_\lambda$,
by (\ref{counting}) we have,
\[
\int_{S_p(1)} N_{g_{\vec{a}}}(r,s,(0))\,\sigma_p(\vec{a}) 
= N_g(r,s,(0)) + O\!\left(\log\frac rs\right)=N_f^{(k)}(r,s,D)+ O\!\left(\log\frac rs\right).
\]
For almost all $\vec a \in S_p(1)$,  
the punctured complex line $\C^*\vec a $ intersects with $\operatorname{Supp} D_\lambda$ transversally.  
There are only finitely many $\lambda \in I$ such that $0 \in D_\lambda$,  denote them by $D_{1},D_{2},\dots,D_{\ell}$.

For almost all $\vec a \in S_p(1)$,
\[
(\C \vec a)\cap \sum_{\lambda \in I} \min\{k,k_\lambda\}D_\lambda
= \sum_{i=1}^\ell \operatorname{ord}_0(D_i)\,\min\{k,k_i\} [0]
+ \sum_{\lambda} \min\{k,k_\lambda\} (P_{\lambda,1}^{\vec a} + P_{\lambda,2}^{\vec a} + \cdots)
\]

\begin{align}\label{intersection}
(f_{\vec a}^*D)^{(k)}=\min\{(f_{\vec a}^*D)_0,\,k\}[0] 
   + \sum_\lambda \min\{k,k_\lambda\}(P_{\lambda,1}^{\vec a} + P_{\lambda,2}^{\vec a} + \cdots)
\end{align}
where \((f_{\vec a}^*D)_0\) denotes the multiplicity of \(f_{\vec a}^*D\) at \(0\) and $\C ^*\vec a$ intersects with $\text{Supp}D_{\lambda}$ at $P_{\lambda,1}^{\vec a},P_{\lambda,2}^{\vec a},\dots$ 

Then we have
\begin{align*}
&\int_{S_p(1)} N_{g_{\vec a}}(r,s,0)\,\sigma_p(\vec a)
= \int_{S_p(1)} \int^r_{s} \frac{dt}{t} 
   \left( \int_{B_{1}(t)\cap \sum \min\{k,k_\lambda\}D_\lambda}1\right) \sigma_p(\vec a) \\[6pt]
&= \int_{S_p(1)} \int^r_{s} \frac{dt}{t} \, 
   \Bigl( \int_{B^*_{1}(t)\cap \sum \min\{k,k_\lambda\}(P_{\lambda,1}^{\vec a} + P_{\lambda,2}^{\vec a} + \cdots)} 1 \Bigr)\sigma_p(\vec a)\\
&\quad+ \sum_{i=1}^\ell \operatorname{ord}_0(D_i)\min\{k,k_i\}\log(\frac{r}{s}).
\end{align*}
On the other hand, by (\ref{intersection}), we have
\begin{align*}
&\int_{S_p(1)} N^{(k)}_{f_{\vec a}}(r,s,D)\,\sigma_p(\vec a) 
= \int_{S_p(1)} \int^r_{s} \frac{dt}{t} 
   \left( \int_{(f_{\vec a}^*D)^{(k)} \cap B_1(t)} 1 \right) \sigma_p(\vec a) \\[6pt]
&= \int_{S_p(1)} \int^r_{s} \frac{dt}{t}
   \left( \int_{B_1(t) \cap \Bigl( \min\{(f_{\vec a}^*D)_0,\,k\}[0] 
   + \sum_\lambda \min\{k,k_\lambda\}(P_{\lambda,1}^{\vec a} + P_{\lambda,2}^{\vec a} + \cdots)\Bigr)} 
   1\right) \sigma_p(\vec a) \\[6pt]
&= \int_{S_p(1)} \int_s^{\,r} \frac{dt}{t}\,
   \left( 
      \int_{B_1^\ast(t) \cap 
      \sum_{\lambda} 
      \min\{k, k_\lambda\}\,\bigl(P^{\vec a}_{\lambda,1} + P^{\vec a}_{\lambda,2} + \cdots\bigr)}
      1 
   \right) 
   \sigma_p(\vec a)  \\
&\quad
  + \int_{S_p(1)}\min\!\left\{(f_{\vec a}^*D)_0,\; k \right\}
  \sigma_p(\vec a)\log\frac{r}{s} .
\end{align*}

Therefore
\[
\Biggl| \int_{S_p(1)} N^{(k)}_{f_{\vec a}}(r,s,D)\,\sigma_p(\vec a)
 - \int_{S_p(1)} N_{g_{\vec a}}(r,s,0)\,\sigma_p(\vec a) \Biggr|
 = O(\log \frac{r}{s}) .
\]

This implies that
\[
N^{(k)}_f(r,D) 
= \int_{S_p(1)} N^{(k)}_{f_{\vec a}}(r,D)\,\sigma_p(\vec a) + O(\log r).
\]
\end{proof}
\section{Green-Griffiths jet differentials and logarithmic  jet differentials}
We  recall the concept of the Green-Griffiths jet differentials and the  Green-Griffiths logarithmic  jet differentials. For references, see \cite{Dem}, \cite{DL01}.

\noindent$\bullet$ {\bf Jet bundle}.   Let $X$ be a complex manifold of dimension $n$.
Let $x\in X$ and let   $J_k(X)_x$
 be  the set of equivalence classes of holomorphic maps  $f:  (\bigtriangleup, 0)\rightarrow (X, x)$, 
where $ \bigtriangleup$ is a disc of unspecified positive radius,
with the equivalence relation $f\sim g$  if and only if 
$f^{(j)}(0)=g^{(j)}(0)$  for $0\leq j\leq k$, when computed in some local coordinate
system of $X$ near $x$. The equivalence class of $f$ is denoted by $j_k(f)$, which is called the $k$-jet of $f$. A $k$-jet $j_k(f)$ is said
to be {\it regular} if $f'(0)\not=0$. 
Set 
$$J_k(X)=\cup_{x\in X} J_k(X)_x$$
and consider the natural projection
$$\pi_k: J_k(X)\rightarrow X.$$
Then $J_k(X)$ is a complex manifold which carries the structure of a holomorphic fiber bundle over $X$, which is called the {\it $k$-jet bundle over $X$}. When $k = 1,$ $J_1(X)$ is canonically isomorphic to the holomorphic tangent bundle  $TX$ of $X$.
On $J_kX$, there is a natural ${\Bbb C}^*$ action defined by, for any $\lambda\in {\Bbb C}^*$ and $j_k(f)\in J_kX$, set 
$$\lambda\cdot j_k(f)=j_k(f_{\lambda})$$
where $f_{\lambda}$ is given by  $t\mapsto f(\lambda t)$. 

\noindent$\bullet$ {\bf Jet Differential}. A {\it jet differential of order $k$} is a holomorphic map $\omega: J_k(X)\rightarrow {\Bbb C}$ which is a polynomial on every fiber, and 
a {\it jet differential of order $k$ and degree $m$} is a holomorphic map $\omega: J_k(X)\rightarrow {\Bbb C}$ which is a polynomial on every fiber such that $$\omega(\lambda\cdot j_k(f))=\lambda^m\omega(j_k(f)).$$
The  Green-Griffiths  sheaf ${\mathcal E}_{k, m}^{GG}$ of order $k$ and degree $m$
is defined as follows: for any open set $U\subset X$, 
$${\mathcal E}_{k, m}^{GG}(U)=\{\mbox{jet differentials of order $k$ and degree $m$ on}~U\}.$$
It is a locally free sheaf, we denote its vector bundle on $X$ as $E_{k, m}^{GG}T^*_X$.
Let $\omega \in {\mathcal E}_{k, m}^{GG}(U)$,  the differentiation 
$d\omega\in {\mathcal E}_{k+1, m+1}^{GG}(U)$ is defined by 
$$d\omega(j_{k+1}(f))= {d\over dt} \omega(j_k(f)(t))\Big|_{t=0}.$$
Note that, formally, we have $d(d^pz_j)=d^{p+1}z_j$ for the local coordinates $(z_1, \dots, z_n)$.

\noindent $\bullet$  {\bf  Logarithmic Jet bundle}.  In the logarithmic setting,  let $D \subset X$ be a normal crossing divisor on $X$. Recall that $D \subset X$ is {\it of normal crossing} 
if, at each $x\in X$,  there exist some local coordinates $z_1,...,z_\ell, z_{\ell+1},..., z_n$  centered at $x$ such that $D$ is locally defined by
$$D = \{z_1\cdots z_\ell = 0\}.$$ 
A holomorphic section $s\in H^0(U,J_k(X))$ over an open subset $U\subset X$ is said to be a \emph{ logarithmic $k$-jet field} if ~$\tilde{\omega}\circ s$ are holomorphic for all sections $\omega\in H^0(U',T_X^*(\log D))$, for all open subsets $U'\subset U$, where $\tilde{\omega}$ are induced maps defined as in \cite{J. Noguchi and J. Winkelmann}. Such logarithmic $k$-jet fields define a
subsheaf of $J_k(X)$, and this subsheaf is itself a sheaf of sections of a holomorphic fiber bundle over $X$, called the \emph{logarithmic $k$-jet bundle over $X$ along $D$}, denoted by $J_k(X;\log D)$.

The group $\mathbb{C}^*$ acts fiberwise on $J_k(X;\log D)$ as follows. For local coordinates 
\[
z_1,\dots,z_{\ell},z_{\ell+1},\dots,z_n\quad{(\ell=\ell(x))}
\]
centered at $x$ in which $D=\{z_1\dots z_{\ell}=0\}$, for any logarithmic $k$-jet field along $D$ represented by some germ $f=(f_1,\dots,f_n)$, if $\varphi_{\lambda}(z)=\lambda z$ is the homothety with ratio $\lambda\in \mathbb{C}^*$, the action is given by
\[
\begin{cases}
\big(\log(f_i\circ \varphi_{\lambda})\big)^{(j)}
=
\lambda^j
\big(\log f_i\big)^{(j)}\circ\varphi_{\lambda} &
\quad
{(1\,\leq\,i\,\leq\,\ell),}
\\
\big(f_i\circ \varphi_{\lambda}\big)^{(j)}
\quad\quad\,=
\lambda^j f_i^{(j)}\circ\varphi_{\lambda}
&
\quad
{(\ell+1\,\leq\,i\,\leq\,n).}
\end{cases}
\]

\noindent $\bullet$  {\bf  Logarithmic Jet differential}. A \emph{logarithmic jet differential} of {\sl order} $k$ and {\sl degree} $m$ at a point $x \in X$ 
is a polynomial $Q(f^{(1)}, \dots, f^{(k)})$ on the fiber over $x$ of 
$J_k(X;\log D)$ enjoying weighted homogeneity:
\[
Q\big(j_k(f \circ \varphi_{\lambda})\big)
=
\lambda^m Q\big(j_k(f)\big).
\]
Denote by $E_{k,m}^{GG} T_X^*(\log D)_x$ the vector space of such polynomials and set
\[
E_{k,m}^{GG} T_X^*(\log D)
:=
\bigcup_{x \in X}
E_{k,m}^{GG} T_X^*(\log D)_x.
\]
By Fa\`a di Bruno's formula \cite{J. Merker}, 
$E_{k,m}^{GG} T_X^*(\log D)$ carries the structure of a vector bundle over $X$, 
called the \emph{logarithmic Green--Griffiths vector bundle}.  

A global section $\mathcal{P}$ of $E_{k,m}^{GG} T_X^*(\log D)$ locally is of the following type in any local coordinates $(z_1, \ldots, z_n)$ on $U$ with 
$D \cap U = \{ z_1 \cdots z_r = 0 \}$,
\begin{align}\nonumber
\mathcal{P}
&=
\sum_{\nu}
\omega_{\nu_{1,1}\cdots \nu_{1,k}\,\cdots\,\nu_{n,1}\cdots \nu_{n,k}}
\;\bigg(
  \frac{dz_1}{z_1}
\bigg)^{\nu_{1,1}}
\cdots
\bigg(
  \frac{d^k z_1}{z_1}
\bigg)^{\nu_{1,k}}
\cdots
\bigg(
  \frac{dz_r}{z_r}
\bigg)^{\nu_{r,1}} \notag \\\nonumber
&\quad \cdots
\bigg(
  \frac{d^k z_r}{z_r}
\bigg)^{\nu_{r,k}}
(dz_{r+1})^{\nu_{r+1,1}}
\cdots
(d^k z_{r+1})^{\nu_{r+1,k}} \notag \\\nonumber
&\quad \cdots
(dz_n)^{\nu_{n,1}}
\cdots
(d^k z_n)^{\nu_{n,k}},
\end{align}
\[
\nu = (\nu_{1,1}, \dots, \nu_{1,k}, \dots, \nu_{n,1}, \dots, \nu_{n,k})
\]
is a $kn$-tuple of non-negative integers with
\[
(\nu_{1,1} + 2\nu_{1,2} + \cdots + k\nu_{1,k})
+ \cdots +
(\nu_{n,1} + 2\nu_{n,2} + \cdots + k\nu_{n,k})
= m,
\]
and
$
\omega_{\nu_{1,1}\cdots \nu_{1,k}\,\cdots\,\nu_{n,1}\cdots \nu_{n,k}}
$
is a locally defined holomorphic function.

\begin{proposition}\label{prop3.1}
Let $X$ be a complex projective manifold of dimension $n$, and let $D$ be a reduced effective divisor on $X$. 
Let $L$ be a line bundle on $X$ with a Hermitian metric $h$. 
Let $\mathcal{P} \in H^0(X, E^{GG}_{k,m}(T_X^*(\log D)) \otimes L)$ be a twisted logarithmic $k$–jet differential. Let $f : \C^p \to X$ be a meromorphic map such that $f(\C^p) \not\subset D$.  
Then
\[
\int_{S^{2p-1}(1)}\int^{2\pi}_0 \log^+ \|\mathcal{P}(j_k(f_{\vec{a}})) (re^{i\theta})\|_{h}\frac{d\theta}{2\pi}\sigma(\vec{a}) 
   \leq  S_f(r, E)
\]
where $E$ is an (arbitrarily) ample divisor on $X$.
\end{proposition}

\begin{proof}
Our proof is partly based on Lemma 4.7.1 in \cite{J. Noguchi and J. Winkelmann}. By virtue of the Hironaka desingularization, we may assume that $D$ is of normal crossing type.  We take an affine covering $\{U_\alpha\}$ of $M$ and rational holomorphic functions 
$(x_{\alpha 1}, \dots, x_{\alpha n})$ on $U_\alpha$ such that
\[
dx_{\alpha 1} \wedge \cdots \wedge dx_{\alpha n}(x) \neq 0, 
\qquad \forall x \in U_\alpha,
\]
and 
\[
D \cap U_\alpha = \{ x_{\alpha 1} \cdots x_{\alpha s_\alpha} = 0 \}.
\]

On every $U_\alpha$ one gets
\[
\mathcal{P}|_{U_\alpha} = P_\alpha \!\left( 
   \frac{d^i x_{\alpha j}}{x_{\alpha j}}, \, d^h x_{\alpha l}
\right),
\qquad 1 \le i,h \le k, \ 
1 \le j \le s_\alpha, \ 
s_\alpha + 1 \le l \le n,
\]
where $P_\alpha$ is a polynomial in the variables described above whose coefficients 
are rational holomorphic functions on $U_\alpha$.

Let $f(z_1,\dots,z_p) = (f_{\alpha 1}(z_1,\dots,z_p), \dots, f_{\alpha n}(z_1,\dots,z_p))$.  
Then
\[
\mathcal{P}(j_k(f_{\vec{a}}))(z) = P_\alpha\!\left( 
   \frac{\frac{\partial^{i} f_{\alpha j,\vec{a}}}{\partial z^i}}{f_{\alpha j,\vec{a}}}, \, \frac{\partial^{h} f_{\alpha l,\vec{a}}}{\partial z^h}
\right).
\]
We have
\begin{align*}
&\frac{\partial f_{\alpha t,\vec{a}}}{\partial z}(z)
  = \sum_{i=1}^p a_i 
    \frac{\partial f_{\alpha t}}{\partial z_i}(\vec{a}z),
\qquad
\frac{\partial^2 f_{\alpha t, \vec{a}}}{\partial z^2}(z)
  = \sum_{i,j=1}^p a_ia_j\frac{\partial^2 f_{\alpha t}}{\partial z_i\partial z_j}(\vec{a}z)
\ldots,\qquad\\
&\frac{\partial^k f_{\alpha t,\vec{a}}}{\partial z^k}(z)
  = \sum_{|\beta|=k} 
    \frac{k!}{\beta!} 
    ( \vec{a})^{\beta}
    \frac{\partial^{|\beta|} f_{\alpha t}}{\partial z^{\beta}}(\vec{a}z), ~~~~1 \le t  \le n.
\end{align*}

Hence, on each local chart $U_\alpha$,
\begin{align*}
&\log^+\!\bigl\|\mathcal{P}(j_k(f_{\vec{a}}))\bigr\|_{h}(z)\\
&\le 
  C_{U_{\alpha}}
  \sum_{\ell=1}^k
  \sum_{|\beta|=\ell}
  \Biggl(
     \sum_{j=1}^{s(\alpha)} 
       \log^+\!\left| 
         \frac{\frac{\partial^{|\beta|} f_{\alpha j}}{\partial z^{\beta}}}{f_{\alpha j}}(\vec{a}z)
       \right|
     + 
     \sum_{h=1}^n 
       \log^+\!\left|
         \frac{\partial^{|\beta|} f_{\alpha h}}{\partial z^{\beta}}(\vec{a}z)
       \right|
  \Biggr)\\ 
&\le 
 C^{\prime}_{U_{\alpha}} 
\sum_{\ell=1}^k 
\sum_{|\beta|=\ell}
   \sum_{j=1}^{n}
    \log^+\!\left| 
         \frac{\frac{\partial^{|\beta|} f_{\alpha j}}{\partial z^{\beta}}}{f_{\alpha j}}(\vec{a}z)
       \right|
+ O(1).    
\end{align*}

Integrating the above inequality over the sphere $S_p(r)$ and applying Theorem A8.1.5 in \cite{book:minru} and Proposition 2.5.20 in \cite{J. Noguchi and J. Winkelmann}, we obtain
\[
\int_{S^{2p-1}(1)}\int^{2\pi}_0 \log^+ \|\mathcal{P}(j_k(f_{\vec{a}}))\|_{h}(re^{i\theta}) \frac{d\theta}{2\pi}\sigma(\vec{a}) 
   \leq  S_f(r, E).
\]
\end{proof}

\section{Holomorphic maps from ${\Bbb C}^p$ into semi-abelian varieties}
\noindent{\it Proof of Bloch's theorem for several variables}.
\begin{proof}
Without loss of generality, we assume $f(0)=0$. Let $X$ be the Zariski closure of $f(\mathbb{C}^p)$. By Corollary 5.1.9 in \cite{J. Noguchi and J. Winkelmann}, $X$ contains only countably many distinct proper semi-abelian subvarieties, say $A_1, A_2, \dots, A_n, \dots$,
then by Lemma~\ref{zero measure}, there exist $K_i \subset S_p(1)$, $i=1,2,\dots$ , such that
$
m_{S_p(1)}(K_i) = 0,
$
and for $\vec{a} \notin K_i$, one has
$
f_{\vec{a}}(\mathbb{C}) \not\subset A_i$. Thus, for $\vec{a} \notin \bigcup_{i=1}^\infty K_i$, we have
$f_{\vec{a}}(\mathbb{C}) \not\subset A_i, i=1,2,\dots.$, it follows that $X=\overline{f_{\vec{a}}(\mathbb{C})}^{\mathrm{Zar}}$ for such (fixed) $\vec{a}$.\\
However, by Bloch's theorem for holomorphic maps from ${\Bbb C}$ into semi-abelian varieties (one-variable-version of Bloch's theorem), $X$ is a semi-abelian subvariety of $A$. This proves our theorem.
\end{proof}

\noindent{\it Proof of Theorem  \ref{SEMIABELIAN}}.
\begin{proof} We follow the arguments of the proof of Theorem 6.4.1 in  \cite{J. Noguchi and J. Winkelmann}. We may assume from the beginning that
\begin{enumerate}[(i)]
    \item $A$ is a semi-abelian variety admitting an algebraic presentation
    \[
        0 \longrightarrow (\mathbb{C}^\ast)^t \longrightarrow A \longrightarrow A_0 \longrightarrow 0,
    \]
    and $\overline{A}$ is a smooth projective equivariant compactification of $A$;
    
    \item $\operatorname{St}(D)^{0} = \{0\}$;
    
    \item the closure $\overline{D}$ is big on $\overline{A}$ and is in good position.
\end{enumerate}

Since $\partial A$ is of normal crossing, consider the logarithmic $k$-th jet bundle
\[
J_k(\overline{A}; \log \partial A)
\]
over $\overline{A}$ along $\partial A$, and a morphism
\[
\psi_k : J_k(\overline{A}; \log \partial A) \to J_k(\overline{A}).
\]
Because of the flat structure of the logarithmic tangent bundle $T(\overline{A}; \log \partial A)$,
\[
J_k(\overline{A}; \log \partial A) \;\cong\; \overline{A} \times \mathbb{C}^{nk}.
\]
Let
\begin{equation}\nonumber
\pi_1 : J_k(\overline{A}; \log \partial A) \cong \overline{A} \times \mathbb{C}^{nk} \;\to\; \overline{A},
\qquad
\pi_2 : J_k(\overline{A}; \log \partial A) \cong \overline{A} \times \mathbb{C}^{nk} \;\to\; \mathbb{C}^{nk}
\end{equation}
be the first and second projections.  
For a $k$-jet $y \in J_k(\overline{A}; \log \partial A)$ we call $\pi_2(y)$ the \emph{jet part} of $y$.
We define
\[
J_k(\overline{D}; \log \partial A) = \psi_k^{-1}\big(J_k(\overline{D})\big).
\]
Note that  $J_k(\overline{D}; \log \partial A)$ is a subspace of $J_k(\overline{A}; \log \partial A)$, which depends in general on the embedding $\overline{D}\hookrightarrow \overline{A}$ (cf. Sect.~4.6.3  in  \cite{J. Noguchi and J. Winkelmann}).   We also note  that $\pi_2(J_k(\overline{D}; \log \partial A))$ is an algebraic subset of $\mathbb{C}^{nk}$, because $\pi_2$ is proper.

Without loss of generality, we assume $f(0)=0$.  Since $f$ is non-degenerate,  $\overline{f(\mathbb{C}^p)}^{\,Zar}=A$. 
By the proof of Bloch's theorem above, we have
$\overline{f_{\vec{a}}(\mathbb{C})}^{\,Zar} = A,
 \text{ for almost every } \vec{a} \in S_p(1).
$
That is, $f_{\vec{a}} : \mathbb{C} \to A$ is algebraically non-degenerate for almost every $\vec{a} \in S_p(1)$.

Fix an $\vec{a}_0 \in S_p(1) \setminus \bigcup_{i=1}^\infty K_i$,  and 
let
\[
j_k(f_{\vec{a}_0}) : \mathbb{C} \to J_k(\overline{A}; \log \partial A) \;\cong\; \overline{A} \times \mathbb{C}^{nk}
\]
be the $k$-th jet lift of $f_{\vec{a}_0}$.  
Let $
Y_{\vec{a}_0, k} = \pi_2\!\big(X_k(f_{\vec{a}_0})\big),$
\medskip
it follows from Lemma~6.4.5 in  \cite{J. Noguchi and J. Winkelmann} that there exists an integral $k_0$ and a polynomial function
$R_0(w)$ in $w \in \mathbb{C}^{nk_0}$ such that
\[
R_0|_{\pi_2 \big(J_k(\overline{D}; \log \partial A)\big)}\equiv0,~~~~~~~ R_0|_{Y_{\vec{a}_0, k_0}}\not\equiv0.
\]
Note here that $R_0$ depends on the vector $\vec{a}_0 \in S_p(1)$.

Denote by $\pi : \mathbb{C}^n \to A$  the universal covering and let
$
\tilde f : z \in \mathbb{C}^p \mapsto \big(\tilde f_1(z),\dots,\tilde f_n(z)\big) \in \mathbb{C}^n$
be a lift of $f$ with $\pi(\tilde f)=f$. Let
$\tilde f_{\vec a}(z)= \bigl(\tilde f_{\vec a,1}(z),\dots,\tilde f_{\vec a,n}(z)\bigr).$
Notice that $R_0|_{Y_{\vec{a}_0, k_0}}\not\equiv0$ is equivalent to $R_0(\tilde f_{\vec{a_0}}',\dots,\tilde f_{\vec{a_0}}^{(k_0)})\not\equiv0$. Let
$H(\vec{a}, z)=R_0(\tilde f_{\vec{a}}',\dots,\tilde f_{\vec{a}}^{(k_0)}).
$
Then there exists $z_0\in\mathbb C$ with $H(\vec{a}_0,z_0)\neq 0$. 
Thus the zero set 
$Z:=\{\vec{a}\in\mathbb C^{p}: H(\vec{a}, z_0)=0\}$ is a proper analytic subset, so if we restrict it to the unit sphere
\(S_{p}(1)\subset\mathbb{C}^p\), then $Z|_{S_{p}(1)}$ is of measure zero in $S_{p}(1)$.
Hence, for almost all $\vec{a}\in S_p(1),$  we have $R_0(\tilde f_{\vec{a}}',\dots,\tilde f_{\vec{a}}^{(k_0)})\not\equiv0 $.

Let $\{U_j\}$ be an affine open covering of $\overline{A}$ such that the line bundle $[\overline{D}]$ is
locally trivial over every $U_j$.  
We take a regular function $\sigma_j$ on each $U_j$ such that $\sigma_j$ is a defining function of $\overline{D}\cap U_j$, i.e.
$
(\sigma_j) = \overline{D}|_{U_j}.$
Fix a Hermitian metric $\|\cdot\|$ on $[\overline{D}]$.  
Then there are positive smooth functions $h_j$ on $U_j$ such that
\[
\|\sigma(x)\| = \frac{|\sigma_j(x)|}{h_j(x)}, \qquad x \in U_j,
\]
which is a well-defined function on $\overline{A}$.
We regard $R_0$ as a regular function on each $U_j \times \C^{nk_0}$.  
Then we have the following equation on every $U_j \times \C^{nk_0}$ 
\begin{equation}\label{eq:6.4.12}
b_{j0}\sigma_j + b_{j1}d\sigma_j + \cdots + b_{j k_0} d^{k_0}\sigma_j = R_0.
\end{equation}
Here each $b_{ji}$ is of the form
\[
b_{ji} = \sum_{\text{finite}} b_{ji l \beta_l}(x) w_l^{\beta_l},
\]
where $b_{ji l \beta_l}(x)$ are regular functions on $U_j$, and $w_l$ are the coordinate functions of $\mathbb{C}^{nk_0}$.  
Thus, we infer that, in every $U_j$
\begin{equation}\nonumber
\frac{1}{\|\sigma\|} = \frac{h_j}{|\sigma_j|}
= \left| h_j b_{j0} + h_j b_{j1}\frac{d\sigma_j}{\sigma_j} + \cdots +
h_j b_{j k_0} \frac{d^{k_0}\sigma_j}{\sigma_j} \right|\frac{1}{R_0}.
\end{equation}
Take relatively compact open subsets $U_j' \Subset U_j$ such that $\bigcup U_j' = \overline{A}$.  
For every $j$ there is a constant $C_j>0$ such that, for $x\in U_j'$,
\[
h_j |b_{ji}| \leq \sum_{\text{finite}} h_j |b_{ji l \beta_l}(x)|\, |w_l|^{\beta_l}
\leq C_j \sum_{\text{finite}} |w_l|^{\beta_l}.
\]
Thus, after enlarging $C_j$ if necessary, there is $d_j>0$ such that, for $z\in \C$ and $f_{\vec{a}}(z)\in U_j'$,
\[
h_j(f_{\vec{a}}(z))\, |b_{ji}(j_{k_0}(f_{\vec{a}})(z))| \leq
C_j\left(1+\sum_{1\leq l\leq n,\;1\leq k\leq k_0}
\big|\tilde f_{\vec{a}, l}^{(k)}(z)\big|\right)^{d_j}
\]
holds for almost all $\vec{a}\in S_p(1)$.

Therefore
\begin{align}\label{eq:6.4.13}
\frac{1}{\|\sigma(f_{\vec{a}}(z))\|}\notag
&= \frac{1}{\big|R_0(\tilde f_{\vec{a}}'(z),\dots,\tilde f_{\vec{a}}^{(k_0)}(z))\big|}
   \left| h_j b_{j0} + h_j b_{j1}\frac{d\sigma_j}{\sigma_j} + \cdots +
   h_j b_{j k_0}\frac{d^{k_0}\sigma_j}{\sigma_j} \right| \\\notag
&\;\leq\; \frac{1}{\big|R_0(\tilde f_{\vec{a}}'(z),\dots,\tilde f_{\vec{a}}^{(k_0)}(z))\big|}
   \sum_{j'=1}^N C_{j'}
   \left(1 + \sum_{1\leq l \leq n,\;1\leq k \leq k_0}
   \big|\tilde f_{\vec{a,l}}^{(k)}(z)\big|\right)^{d_{j'}} \notag \\\notag
&\qquad \times \Bigg(1 + 
   \left|\frac{d\sigma_{j'}}{\sigma_{j'}}(j_1(f_{\vec{a}})(z))\right| + \cdots +
   \left|\frac{d^{k_0}\sigma_{j'}}{\sigma_{j'}}(j_{k_0}(f_{\vec{a}})(z))\right|\Bigg). \notag
\end{align}

It follows that
\begin{align*}
m_{f_{\vec a}}(r,\overline D)
&=
\int_0^{2\pi}
   \log \frac{1}{\bigl\|\sigma\!\bigl(f_{\vec a}(re^{i\theta})\bigr)\bigr\|}
   \,\frac{d\theta}{2\pi}
   -
   \int_0^{2\pi}
   \log \frac{1}{\bigl\|\sigma\!\bigl(f_{\vec a}(e^{i\theta})\bigr)\bigr\|}
   \,\frac{d\theta}{2\pi}\\[6pt]
&\le 
\int_0^{2\pi}
   \log \frac{1}{\bigl|
      R\!\bigl(\tilde f_{\vec a}',\dots,\tilde f_{\vec a}^{(k_0)}\bigr)
      (re^{i\theta})
   \bigr|}
   \,\frac{d\theta}{2\pi} \\
&+\, O\!\Bigg(
  \sum_{\substack{1\le l\le n\\ 1\le k\le k_0}}
  \int_0^{2\pi}
     \log^{+}\!\bigl|
        \tilde f^{(k)}_{\vec a,l}(re^{i\theta})
     \bigr|
     \,\frac{d\theta}{2\pi}
  +\;
  \sum_{\substack{1\le j\le N\\ 1\le k\le k_0}}
  \int_0^{2\pi}
     \log^{+}\!\left|
        \frac{
           (\sigma_j\!\circ f_{\vec a})^{(k)}
        }{
           \sigma_j\!\circ f_{\vec a}
        }
        (re^{i\theta})
     \right|
     \,\frac{d\theta}{2\pi}
\Bigg)  \\
&-
   \int_0^{2\pi}
   \log \frac{1}{\bigl\|\sigma\!\bigl(f_{\vec a}(e^{i\theta})\bigr)\bigr\|}
   \,\frac{d\theta}{2\pi}
+ O(1).
\end{align*}

Since $R_0\!\big(\tilde f_{\vec a}',\dots,\tilde f_{\vec a}^{(k_0)}\big)$ is a holomorphic function on $\C^p\times\C$, it can be written as $z^{n_0}\sum^\infty_{i=0} a_i(\vec a)z^i$, with $a_0\not\equiv0$ on $\C^p$.

Therefore
\begin{align*}\
\int^{2\pi}_0\log\frac{1}{|R_0\!\big(\tilde f_{\vec a}',\dots,\tilde f_{\vec a}^{(k_0)}\big)(re^{i\theta})|}\frac{d\theta}{2\pi}
&=
\int^{2\pi}_0\log\frac{1}{|\sum^\infty_{i=0} a_i(\vec a)z^i|}(re^{i\theta})\frac{d\theta}{2\pi}-{n_0}\log r\\
&\leq 
\int^{2\pi}_0\log\frac{1}{|a_0(\vec a)|}\frac{d\theta}{2\pi}-{n_0}\log r\\
&=\log\frac{1}{|a_0(\vec a)|}-{n_0}\log r.
\end{align*}
Integrating the above inequality over $S_p(1)$, and applying Proposition \ref{2.5}, Proposition \ref{prop3.1}, Theorem A8.1.5 in \cite{book:minru}, and Lemma 3.8 in \cite{NWY}, we get
\begin{align*}
\int_{S_p(1)} \log \frac{1}{|a_0(\vec a)|} \, \sigma(\vec a) &< \infty,\\[6pt]
\int_{S_p(1)}
 \int_0^{2\pi}
   \log^{+}\!\left|
      \frac{(\sigma_j\!\circ f_{\vec a})^{(k)}}{\sigma_j\!\circ f_{\vec a}}
      (re^{i\theta})
   \right|
   \,\frac{d\theta}{2\pi}
\, \sigma(\vec a)
&\le S_f\!\bigl(r, \overline{D}\bigr),
\intertext{and}
\int_{S_p(1)} \int_0^{2\pi}
   \log^{+}\!\bigl|\tilde f^{(k)}_{\vec a,l}(re^{i\theta})\bigr|
   \,\frac{d\theta}{2\pi}\,  \sigma(\vec a)
&\le S_f\!\bigl(r, \overline{D}\bigr).\\[6pt]
\end{align*}

Observe that $\operatorname{ord}_z  f_{\vec a}^*D > k$ if and only if 
$j_k( f_{\vec a})(z) \in J_k(\overline D; \log \partial A)$, from this and (\ref{eq:6.4.12}), we infer that
\[
\operatorname{ord}_z  f_{\vec a}^*D - \min\{\operatorname{ord}_z  f_{\vec a}^*D, k_0\} 
\leq \operatorname{ord}_z\bigl(R_0(\tilde  f_{\vec a}', \ldots, \tilde  f_{\vec a}^{(k_0)})\bigr).
\]
Thus, we have, after taking integration, that
\[
N_{f_{\vec a}}(r,  D) - N_{f_{\vec a}}^{(k_0)}(r, D)
\leq N\bigl(r, R_0(\tilde  f_{\vec a}', \ldots, \tilde  f_{\vec a}^{(k_0)}),0\bigr).
\]

It follows that
\begin{align*}
&~~~~~~~N\bigl(r, R_0(\tilde  f_{\vec a}', \ldots, \tilde  f_{\vec a}^{(k_0)}),0\bigr)\\
&=
\int^{2\pi}_0\log{|R_0\!\big(\tilde f_{\vec a}',\dots,\tilde f_{\vec a}^{(k_0)}\big)|}(re^{i\theta})\frac{d\theta}{2\pi}-\int^{2\pi}_0\log{|R_0\!\big(\tilde f_{\vec a}',\dots,\tilde f_{\vec a}^{(k_0)}\big)|}(e^{i\theta})\frac{d\theta}{2\pi}\\
&\leq
\int^{2\pi}_0\log^+{|R_0\!\big(\tilde f_{\vec a}',\dots,\tilde f_{\vec a}^{(k_0)}\big)|}(re^{i\theta})\frac{d\theta}{2\pi}-\log|a_0(\vec a)|\\
&\leq 
 O\!\left(\sum_{\substack{1\le l\le n\\ 1\le k\le k_0}}
  \int_0^{2\pi}
   \log^{+}\!\bigl|\tilde f^{(k)}_{\vec a,l}(re^{i\theta})\bigr|
   \,\frac{d\theta}{2\pi}\,\right)-\log|a_0(\vec a)|.
\end{align*}
Integrating the above inequality over $S_p(1)$, we get
\begin{align*}
N_f(r, D) \leq N_f^{(k_0)}(r,D) + S_f\bigl(r, \overline{D}).
\end{align*}
Combining these together, it gives
\begin{align*}
T_f(r, \overline{D})\leq N^{(k_0)}_{f}(r, D)+S_f(r, \overline{D}) .    
\end{align*}
\end{proof}
\section{Meromorphic maps from ${\Bbb C}^p$ into projective varieties}

\begin{proposition}\label{jet exsits}
Let $c$ be a positive integer with $c \ge 5n-1$, and let
$D \subset \mathbb{P}^n(\mathbb{C})$ be a generic smooth hypersurface of degree
$d \ge 15(c+2)n^{n}$.
Let $f:\mathbb{C}^p \to \mathbb{P}^n(\mathbb{C})$ be a meromorphic,
algebraically nondegenerate map.
Then, for jet order $k=n$ and for weighted degrees $m \gg d$,
there exist an integer $0\le \ell \le m$ and a global logarithmic jet differential
\[
{\mathcal P} \in H^0\!\left(\mathbb{P}^n(\mathbb{C}),
E^{\mathrm{GG}}_{n,m}\, T^*_{\mathbb{P}^n(\mathbb{C})}(\log D) \otimes
\mathcal{O}_{\mathbb{P}^n(\mathbb{C})}\big(-cm+\ell(5n-2)\big)
\right)
\]
such that
$
{\mathcal P}\big(j_n(f_{\vec a})\big)\not\equiv 0
$ for almost every $\vec a \in S_{p}(1)$.
\end{proposition}

\begin{proof}
Since $f$ is algebraically nondegenerate, its radial derivative
\[
\frac{\partial f}{\partial r}(z)
=
\sum_{i=1}^p \frac{z_i}{\|z\|}\,\frac{\partial f}{\partial z_i}(z),
\qquad z=(z_1,\dots,z_p)\in \mathbb{C}^p\setminus\{0\},
\]
does not vanish identically; hence the set
$
\bigl\{ z\in\mathbb{C}^p\setminus\{0\} : \tfrac{\partial f}{\partial r}(z)=0 \bigr\}
$
has zero measure in $\mathbb{C}^p$.

By Proposition~2.1 in \cite{HVX}, there exists a nonzero logarithmic jet differential
\[
{\mathcal P}_s \in H^0\!\left(\mathbb{P}^n(\mathbb{C})\times\{s\},
E^{\mathrm{GG}}_{n,m}\, T^*_{\mathbb{P}^n(\mathbb{C})}(\log D)\otimes
\mathcal{O}_{\mathbb{P}^n(\mathbb{C})}(-cm)\right).
\]
This implies that $f^{-1}\!\big(\{{\mathcal P}_s=0\}\cup D\big)$ has zero measure in $\mathbb{C}^p$.
Consequently
$
f^{-1}\!\big(\{{\mathcal P}_s=0\}\cup  D\big)\cup
\bigl\{ \tfrac{\partial f}{\partial r}=0 \bigr\}
$
has zero measure in $\mathbb{C}^p$, and we may choose $z_0\in\mathbb{C}^p\setminus\{0\}$ with
$
f(z_0)\notin \{{\mathcal P}_s=0\}\cup D \text{~and}
\
\frac{\partial f}{\partial r}(z_0)\ne 0.
$

Let $\vec{a_0}:=z_0/\|z_0\|\in S^{2p-1}(1)$ and consider the restriction
$f_{\vec{a_0}}:\mathbb{C}\to\mathbb{P}^n(\mathbb{C})\times\{s\}$, $f_{\vec{a_0}}(z)=f(z\vec{a_0})$. By the semi-continuity theorem, $\mathcal P_s$ extends to a holomorphic family $\mathcal{P}$ of nonzero logarithmic jet differentials parametrized by points of a Zariski open neighborhood $U_s$ of $s$ in $S$. Following the argument in \cite{HVX} (p.~668), there exist global slanted vector fields
$v_1,\dots,v_\ell$ (for some $0\le \ell\le m$) such that
\[
\big((v_1\cdots v_\ell)\mathcal{P}\big)\big(j_n(f_{\vec{a_0}})\big)\not\equiv 0 .
\]
Hence
\[
\tilde{\mathcal{P}}_{s}
:=(v_{1}\cdots v_{\ell})\, 
\mathcal{P}\big|_{\mathbb{P}^{n}(\mathbb{C})\times\{s\}}
\in
H^{0}\!\left(
  \mathbb{P}^{n}(\mathbb{C}),
  E^{\mathrm{GG}}_{n,m}T^{*}_{\mathbb{P}^{n}(\mathbb{C})}(\log D)
  \otimes
  \mathcal{O}_{\mathbb{P}^{n}(\mathbb{C})}\!\big(-cm+\ell(5n-2)\big)
\right)
\]
satisfies $\tilde{\mathcal{P}}_{s}\big(j_n(f_{\vec{a_0}})\big)\not\equiv 0$. Similar with the argument in the proof of the Theorem \ref{SEMIABELIAN}, we have
\[
\tilde{\mathcal{P}}_{s}\big(j_n(f_{\vec{a}})\big)\not\equiv 0 \quad \text{for almost every } \vec a\in S_p(1).
\]
\end{proof}

\noindent{\it Proof of Theorem \ref{smtXie}}
	\begin{proof}
   Since 
${\mathcal P}\big(j_k(f_{\vec a})\big)\not\equiv 0$ for some $\vec{a}\in S_{p}(1)$, 
it follows that ${\mathcal P}\big(j_k(f_{\vec a})\big) \not\equiv 0$ for almost all 
$\vec{a}\in S_{p}(1)$.  
Choose an $\vec{a}$ with 
$f_{\vec{a}}(\mathbb{C}) \not\subset \operatorname{Supp}(D)$ 
and ${\mathcal P}\big(j_k(f_{\vec a})\big) \not\equiv 0$.  
Following the proof of Theorem~1.6 in \cite{CaiRuYang}, we obtain, 
in the sense of currents,
\[
dd^{c} \log \|P(j_{k}(f_{\vec a}))\|^{2}_{h^{-1}}\ge
    \widetilde{m}\, f_{\vec a}^{*}c_{1}(A) - m (f_{\vec a}^*D)^{(1)}
\]
where $h$ is a Hermitian metric on $A^{\widetilde{m}}$.
Hence, by applying the Green–Jensen formula,
we obtain
\begin{align}\label{eq:circle-avg}
\widetilde{m}\, T_{f_{\vec a}}(r,A)- m\, N^{(1)}_{f_{\vec a}}(r, D)
    \le
    \int_{0}^{2\pi}
        \log \|P(j_{k}(f_{\vec a})(re^{i\theta})) )\|_{h^{-1}}
        \, \frac{d\theta }{2\pi}
    +
    O(1).
\end{align}
Finally, integrating \eqref{eq:circle-avg} over $S_p(1)$, we get
\begin{align*}\label{eq:Sp-integration}\nonumber
\widetilde{m}\,T_f(r,A)-m\,N_f^{(1)}(r,D)
&\le\;\int_{S_p(1)}
\left(
  \int_{0}^{2\pi}
  \log\|\mathcal{P}\!\big(j_k(f_{\vec a})(re^{i\theta})\|_{h^{-1}}
  \,\frac{d\theta}{2\pi}
\right)
d\sigma_p(\vec a) \\
&~~~~~~~~~~~~~~~~~+O(\log r).
\end{align*}
By Proposition \ref{prop3.1}, we have
\[
\widetilde{m}\, T_f(r,A) - m\, N_f^{(1)}(r, D) 
   \le_{\mathrm{exc}} S_f(r,A).
\]
\end{proof}

\noindent{\it Proof of Theorem \ref{GENERIC}}
\begin{proof}
Choose
$
\widetilde{m} = mc - \ell(5n-2) \ge m(c - 5n + 2) \ge m,
$
where $c$ is a positive integer satisfying $c \ge 5n-1$. 
Applying Theorem~\ref{smtXie} in the case 
$X=\mathbb{P}^{n}(\mathbb{C})$, with $D$ a generic hypersurface and 
$A=\mathcal{O}(1)$, we have
\[
\widetilde{m}\, T_f(r) - m\, N_f^{(1)}(r, D) 
   \le_{\mathrm{exc}} S_f(r).
\]
Then, 
\[
T_f(r) \le_{\mathrm{exc}} \frac{1}{c - 5n + 2}\, N_f^{(1)}(r, D) 
   + S_f(r).
\]

If we choose $c = 5n-1$, we get
\[
T_f(r) \le_{\mathrm{exc}} N_f^{(1)}(r, D) 
   + S_f(r).
\]
\end{proof}


\begin{thebibliography}{Qua19}
 
 
\bibitem{Brotbek}
D. Brotbek,
\newblock On the hyperbolicity of general hypersurfaces.
\newblock {\em Publ. Math. Inst. Hautes Études Sci.}, {\bf 126}, 1--34, 2017.

 
 
\bibitem{BD}
D. Brotbek and Y. Deng,
\newblock Kobayashi hyperbolicity of the complements of general hypersurfaces of high degree.
\newblock {\em Geom. Funct. Anal.}, {\bf 29}, 690--750, 2019.



\bibitem{CaiRuYang}
Q. Cai, M. Ru and J. Yang,
\newblock  A non-integrated defect relation for holomorphic maps into algebraic varieties,
\newblock {arXiv:2501.00827}, 2025.


\bibitem{CaiRuYang2}
Q. Cai, M. Ru and J. Yang,
\newblock  Defect relations for meromorphic maps
 of complete K$\ddot{A}$hler manifolds into projective varieties.
 \newblock {\em Complex Anal Synerg}, {\bf  11}, 14, 2025.



\bibitem{Dem}
J. P. Demailly,
\newblock Algebraic Criteria for Kobayashi hyperbolic varieties and
jet differentials.
\newblock  {\em Proc. Symp. Pur. Math. Amer. Math. Soc.}, {\bf 62}, Part 2, 285--360, 1995.


\bibitem{DL01}
G.-E.~Dethloff and S.~S.-Y.~Lu,
\newblock Logarithmic jet bundles and applications.
\newblock {\em Osaka J. Math.}, {\bf 38}(1):185--237, 2001.

\bibitem{rosseauimpa}
S.~Diverio and E.~Rousseau,
\newblock { Hyperbolicity of projective hypersurfaces},  Vol.~5 of {\em
  IMPA Monographs},
\newblock Springer, Cham, 2016.

\bibitem{Ete}
A. Etesse,
\newblock { Geometric Generalized Wronskians: Applications in intermediate hyperbolicity and foliation theory},
\emph{ International Mathematics Research Notices,} Vol. 2023, Issue 10,  8251-8310, 2023.

\bibitem{AG}
R. Hartshorne, Algebraic geometry.
\emph{ Graduate Texts in Mathematics, }No. 52. Springer-Verlag, NewYork-Heidelberg, 1977. xvi+496 pp.


\bibitem{HuYang}
P.-C. Hu and C.-C. Yang, Jet bundles and its applications in value distribution of holomorphic mappings.
\emph{Adv. Complex Anal. Appl,} {\bf 3}, 281--319, Kluwer Acad. Publ, Boston, MA, 2004.


\bibitem{HVX}
D. T.  Huynh, D. V. Vu and S. Y.  Xie,
\newblock 
Entire holomorphic curves into projective spaces intersecting a generic hypersurface of high degree.
\newblock {\em Annales de l’Institut Fourier}, {\bf 69}(2): 653--671, 2019.

\bibitem{J. Merker}
J. Merker,
\newblock Low pole order frames on vertical jets of the universal hypersurface.
\newblock {\em Ann. Inst. Fourier.} {\bf 59}, no.~3, 1077--1104, 2009.




\bibitem{J. Noguchi and J. Winkelmann}
J. Noguchi and J. Winkelmann,
\newblock {\em Nevanlinna theory in several complex variables and Diophantine approximation}. 
\newblock  Grundlehren der mathematischen Wissenschaften, 350, Springer, Tokyo, 2014.


\bibitem{NWY}
J. Noguchi, J. Winkelmann and K. Yamanoi,
\newblock {The second main theorem for holomorphic curves into semi-abelian varieties}. 
\newblock { \em Acta Math.} {\bf 188}, no.~1, 129--161, 2002.


\bibitem{NWY2}
J. Noguchi, J. Winkelmann and K. Yamanoi,
\newblock { The second main theorem for holomorphic curves into semi-abelian varieties. II}. 
\newblock  {\em Forum Math.} {\bf 20} , no.~3, 469--503, 2008.

\bibitem{PR}
G. Pacienza and E. Rousseau, 
\newblock { Generalized Demailly-Semple jet bundles and
holomorphic mappings into complex manifolds}. 
\newblock{\em J. Math. Pures Appl.} (9), 96(2):109–
134, 2011.

\bibitem{RY}
Riedl, E., Yang, D.,
\newblock {Applications of a Grassmannian technique to hyperbolicity, Chow equivalency, and Seshadri constants}. 
\newblock   {\em J. Algebraic Geom}, no.~1, 1--12, 2022.




\bibitem{ru_annals}
M. Ru,
\newblock Holomorphic curves into algebraic varieties.
\newblock {\em Ann. of Math.}, {\bf 169}(1), 255--267, 2009.



\bibitem{book:minru}
M. Ru,
\newblock {\em Nevanlinna theory and its relation to Diophantine
  approximation}.
\newblock Second Edition, World Scientific, River Edge, Singapore, 2021.

\bibitem{book:minru2}
M. Ru,
\newblock {\em Minimal Surfaces through Nevanlinna theory}.
\newblock Studies in Mathematics 92, Water de  Gruyter GmbH, Berlin/Boston, 2023.



\bibitem{DMR}
S. Diverio, J. Merker and E. Rousseau,
\newblock {Effective algebraic degeneracy}.
\newblock {\em Invent. Math}. {\bf 180}, no.~1, 161--223, 2010.

\bibitem{HCG}
Y.T. Siu,
\newblock {Hyperbolicity in complex geometry}.
\newblock  In: The legacy of Niels Henrik Abel. Springer, Berlin, 2004, pp. 543–566.


\bibitem{Siu2015}
Y.T. Siu,
\newblock Hyperbolicity of generic high-degree hypersurfaces in complex projective space.
\newblock{\em  Invent. Math}, {\bf 202}(3), 1069–1166, 2015. 

\bibitem{Siu-Y3}
Y. T. Siu  and S. K. Yeung,
\newblock Defects for ample divisors of abelian 
varieties, Schwarz lemma, and hyperbolic hypersurfaces of low degree.
\newblock {\em Amer. J. of Math.} {\bf 119}, 1139--1172, 1997.




\bibitem{W. Stoll}
W. Stoll,
\newblock {\em Holomorphic functions of finite order in several complex variables}.
\newblock Conference Board of the Mathematical Sciences Regional Conference Series in Mathematics, No. 21, Amer. Math. Soc., Providence, RI, 1974.

\bibitem{W. Stoll2}
W. Stoll,
\newblock {Die beiden Haupts~itze der Wertverteilungstheorie (I)}.
\newblock {\em Acta Math}. {\bf 90}, 1–115, 1953. 





\bibitem{P.-M. Wong}
P.-M. Wong,
\newblock {Holomorphic mappings into abelian varieties}.
\newblock {\em Amer. J. of Math.}, {\bf 102}: no.~3, 493--502, 1980.


\bibitem{Yamanoi2}
K. Yamanoi,
\newblock Kobayashi hyperbolicity and higher-dimensional Nevanlinna theory.
\newblock In: {\em Geometry and Analysis on Manifolds, Progress in Mathematics}, Birkhauser/Springer, Vol. 308, 209--273, 2015.

\end{thebibliography}
\end{document}